\documentclass{amsart}

\usepackage{amsmath, amsthm, amssymb, graphicx, color}
\usepackage[normalem]{ulem}

\newtheorem{theorem}{Theorem}[section]
\newtheorem{lemma}[theorem]{Lemma}
\newtheorem{proposition}[theorem]{Proposition}
\newtheorem{corollary}[theorem]{Corollary}
\newtheorem{definition}[theorem]{Definition}

\theoremstyle{remark}
\newtheorem{remark}{Remark}

\def\pq{(\!(q)\!)}
\def\bq{[\![q]\!]}

\author{Oleg Lazarev}
\address{Department of Mathematics, Stanford University, Stanford, CA 94305, USA}
\email{olazarev@math.stanford.edu}

\author{Matthew S. Mizuhara}
\address{Department of Mathematics, Penn State University, State College, PA 16802, USA}
\email{msm344@psu.edu}

\author{Benjamin Reid}
\address{Department of Mathematics, University of Oregon, Eugene, OR, 97403, USA}
\email{bwr@uoregon.edu}

\author{Holly Swisher}
\address{Department of Mathematics, Oregon State University, Corvallis, OR 97331, USA}
\email{swisherh@math.oregonstate.edu}

\title{An extension of a proof of the Ramanujan congruences for multipartitions}
\title[A proof of the Ramanujan congruences for multipartitions]{extension of a proof of the Ramanujan congruences for multipartitions}

\thanks{The authors were supported in part by the NSF grant DMS-0852030.}

\subjclass[2010]{Primary 11P83}
\keywords{partitions, multipartitions, Ramanujan, congruences}

\begin{document}

\maketitle

\begin{abstract}
Recently Lachterman, Schayer, and Younger published {an} elegant proof of the Ramanujan congruences for the partition function $p(n)$.  Their proof uses only the classical theory of modular forms as well as a beautiful result of Choie, Kohnen, and Ono, without need for Hecke operators.  In this paper we give a method for generalizing Lachterman, Schayer, and Younger's proof to include Ramanujan congruences for multipartition functions $p_k(n)$, and Ramanujan congruences for $p(n)$ modulo certain prime powers.    
\end{abstract}

\section{Introduction and Statement of Results}

The subject of partitions has a long fascinating history, including connections to many areas of mathematics, and mathematical physics.  For example, \cite{Andrews} and \cite{A-O-Notices} provide a glimpse into the history of partitions.  In particular, the generalization of partitions to $k$-component multipartitions, also known as $k$-colored partitions, is a rich subject in its own right (see \cite{Andrews-Multipartitions} for a nice survey).  We will begin by reviewing partitions and multipartitions.
 
\subsection{Partitions}
We recall that a {\it{partition}} of a positive integer $n$ is defined as a nonincreasing sequence of positive integers called {\it{parts}}, that sum to $n$ (often written as a sum).  For $n=0$ we define the empty set as the unique ``empty partition" of $0$.  We write  $|\lambda|=n$ to denote that $\lambda$ is a partition of $n$.  For example, the following gives all the partitions $\lambda$ such that $|\lambda|=5$:
$$
5 = 4+1 = 3+2 = 3+1+1 = 2+2+1 = 2+1+1+1= 1+1+1+1+1.
$$

\noindent The {\it{partition function}} $p(n)$ counts the total number of partitions of $n$.  In order to define $p(n)$ on all integers we further define that $p(n)=0$ when $n<0$.  We see from our example above that $p(5)=7$.

The generating function for $p(n)$ has the following infinite product form, due to Euler:
$$
\sum_{n=0}^\infty p(n)q^n=\prod_{n=1}^\infty\frac{1}{(1-q^n)}.
$$

\subsection{Multipartitions}\label{multipartitions}
One natural generalization of partitions is by the following.  Define \color{black} a {\it{$k$-component multipartition}} of a nonnegative integer $n$ to be a $k$-tuple of partitions $\lambda=(\lambda_1, \ldots, \lambda_k)$ such that 
$$
\sum_{i=1}^k|\lambda_i|=n.
$$
We will write $|\lambda|_k = n$ if $\lambda$ is a $k$-component multipartition of $n$.  The following gives all the $2$-component multipartitions of $3$, i.e., all $\lambda$ such that $|\lambda|_2 = 3$:
$$
(3,\emptyset), (2+1, \emptyset), (1+1+1, \emptyset), (2,1), (1+1,1), (1,2), (1,1+1), (\emptyset,3), (\emptyset, 2+1), (\emptyset, 1+1+1).
$$
We define $p_k(n)$ as the number of $k$-component multipartitions of $n$, again defining $p_k(n)=0$ for $n<0$.  From our example above we see that $p_2(3)=10$.
{\begin{remark}  We note that ordering matters in this definition, in that a rearrangement of components $\lambda_i$ may yield a distinct multipartition.  In addition, we draw attention to the fact that since the empty set is a partition of $0$ some $\lambda_i$ may equal $\emptyset$.
\end{remark} }

The generating function for $p_k(n)$ follows from the generating function for $p(n)$ by taking the $k$th power of the product form.  Namely,
\begin{equation}\label{genfn}
\sum_{n=0}^\infty p_k(n)q^n=\prod_{n=1}^\infty\frac{1}{(1-q^n)^k}.
\end{equation}

\subsection{Ramanujan Congruences}
{Among} the most celebrated results in partition theory {are} the following congruences of Ramanujan for $p(n)$.  For all integers $n\geq 0,$
\begin{eqnarray*}
p(5n+4) & \equiv & 0 \pmod 5\\
p(7n+5) & \equiv & 0 \pmod 7\\
p(11n+6) & \equiv & 0 \pmod {11}.\\
\end{eqnarray*}
 
Work of Ono and Ahlgren \cite{Ono-distribution}, \cite{A-O}, \cite{Ahlgren} has shown that for any $m$ coprime to $6$ there exist infinitely many nonnested arithmetic progressions for which $p(an+b)\equiv 0 \pmod {m}$.  However, it has been shown by Ahlgren and Boylan \cite{A-B} that the three Ramanujan congruences above are the only congruences for $p(n)$ of the form 
$$
p(\ell n+b) \equiv 0 \pmod{\ell},
$$
for $\ell$ prime.

Many congruences for multipartitions have been proven as well, see Andrews \cite{Andrews-Multipartitions}, Gandhi \cite{Gandhi}, Atkin \cite{Atkin}, and Kiming and Olsson \cite{K-O}.  Recent work by Folsom, Kent, and Ono \cite{F-K-O} discusses {some of} the underlying reasons for such congruences.

Using notation of Kiming and Olsson \cite{K-O}, for a prime $\ell\geq 5$, an integer $a$, and positive integers $k,m$, we say {\it{there is a congruence at}} $(\ell^m,k,a)$ if for all $n\geq 0$, 
$$
p_k(\ell^mn+a) \equiv 0 \pmod{\ell^m}.
$$

There have been a number of proofs of the Ramanujan congruences.  The proof by Lachterman, Schayer, and Younger in \cite{LSY} is notable because it uses only the classical theory of modular forms, without relying on Hecke operators.  In this paper we extend their proof to include Ramanujan congruences for multipartitions modulo prime powers.  In particular, we prove the following theorem.

\begin{definition}
Let $\ell\geq 5$ be prime, and let $k,m$ be positive integers.  Then we define
\[
\delta_{k,\ell,m}:= \frac{k(\ell^{2m}-1)}{24}.
\]
{We note that the $\delta_{k,\ell,m}$ are all positive integers,} since $\ell^2\equiv 1 \pmod{24}$ for all primes $\ell\geq 5$.  
\end{definition}

\begin{theorem}\label{main}
Let $\ell\geq 5$ be prime, and let $k,m$ be positive integers such that $k\equiv -4 \pmod{\ell^{m-1}}$.  If {for each $1\leq r <m$ there is a congruence at $(\ell^r,k,-\delta_{k,\ell,m})$,} then there is a congruence at $(\ell^m,k,-\delta_{k,\ell,m})$ if and only if 
\[
p_k(\ell^mn-\delta_{k,\ell,m})\equiv 0 \pmod{\ell^m}
\]
for all $0\leq n < \frac{k \ell^m +2\ell+2}{24}$.
\end{theorem} 

\begin{remark}\label{rem1}
Note that {whenever  $1 \le r \le m$ we have $\delta_{k,\ell, m} \equiv \delta_{k,\ell, r} \pmod{\ell^r}$.  Thus there is a congruence at $(\ell^r,k,-\delta_{k,\ell,m})$ if and only if there is a congruence at $(\ell^r,k,-\delta_{k,\ell, r})$.  Here we are using the fact that $\delta_{k,\ell,r}$ is a positive integer for each $r$, and that we defined $p_k$ to take the value $0$ when evaluated at any negative integer.  More specifically, the condition that $p_k(\ell^r n -\delta_{k,\ell,r})\equiv 0 \pmod{\ell^r}$ for $0\leq n \leq N$ is equivalent to the condition that} {$p_k(\ell^r n -\delta_{k, \ell,m}) \equiv 0 \pmod{\ell^r}$ for $0 \leq n \leq N+ \frac{\delta_{k,\ell,m} - \delta_{k,\ell,r}}{\ell^r}$.} 
In particular, in Theorem \ref{main} {one can} replace the condition that there is a congruence at $(\ell^r,k,-\delta_{k,\ell,m})$ for all $1\leq r <m$  with the equivalent condition that there is a congruence at $(\ell^r,k,-\delta_{k,\ell,r})$ for all $1\leq r <m$. 
\end{remark}

{We observe that the $m=1$ case of Theorem \ref{main} has trivial conditions.  In particular it yields that when $\ell\geq 5$ is prime and $k$ is a positive integer, there is a  congruence at $(\ell, k, - \delta_{k,\ell,1})$ if $p_k(\ell n - \delta_{k,\ell,1})\equiv 0 \pmod{\ell}$ for all $0\leq n < \frac{k\ell + 2\ell + 2}{24}$.  We also observe that for a general positive integer $m$, the condition $k\equiv -4 \pmod{\ell^{m-1}}$ in Theorem \ref{main} implies that $k\equiv -4 \pmod{\ell^r}$ for any $1\leq r < m-1$ as well.   Thus, fixing a prime $\ell\geq 5$, and positive integers $k,m$ such that $k\equiv -4 \pmod{\ell^{m-1}}$, suppose we know for each $1\leq r \leq m$ that $p_k(\ell^r n - \delta_{k,\ell,r}) \equiv 0 \pmod {\ell^r}$ for all $0\leq n < \frac{k\ell^r + 2\ell + 2}{24}$.  Then, the $m=1$ case of Theorem \ref{main} implies that there is a congruence at $(\ell, k, - \delta_{k, \ell, 1})$, and so by Remark \ref{rem1} there is a congruence at $(\ell, k, - \delta_{k, \ell, 2})$.  Thus the conditions of the $m=2$ case of Theorem \ref{main} are satisfied, so we can conclude in fact that there is also a congruence at $(\ell^2, k, - \delta_{k, \ell, 2})$, and again by Remark \ref{rem1} a congruence at $(\ell^2, k, - \delta_{k, \ell, 3})$.  Continuing inductively, we obtain that there is in fact a congruence at $(\ell^r, k, - \delta_{k, \ell, r})$ for each $1\leq r \leq m$.  In particular, we have the following corollary which gives a finite condition for the existence of a congruence at $(\ell^m, k,-\delta_{k,\ell,m})$.
}

\begin{corollary}\label{nice}
Let $\ell\geq 5$ be prime, and let $m,k$ be positive integers such that $k\equiv -4 \pmod{\ell^{m-1}}$. If for each $1\leq r\leq m$,
\[
p_k(\ell^rn-\delta_{k,\ell,r})\equiv 0 \pmod{\ell^r},
\]
for all $0\leq n < \frac{k \ell^{r} + 2\ell +2}{24}$,
then there is a congruence at 
$(\ell^r, k, -\delta_{k,\ell, r})$ for all $1\le r\le m$. 
\end{corollary}

\begin{remark}\label{rem3}
{In light of Remark \ref{rem1} one can rewrite the conditions in Corollary \ref{nice} as }
$p_k(\ell^rn-\delta_{k,\ell,m})\equiv 0 \pmod{\ell^r}$ for all {$0 \le n <  \frac{k \ell^{2m - r}+ 2\ell +2}{24}$.}
Likewise, the conclusion can be replaced with congruences at $(\ell^r, k, -\delta_{k,\ell, m})$  for all $1\le r\le m$.
\end{remark}

In Section \ref{prelim}, we establish some preliminaries which enable us to prove Theorem \ref{main}.  In addition {we prove a generalization of a lemma of Kiming and Olsson \cite{K-O} in Lemma \ref{KOgeneralization},} which combined with Corollary \ref{nice} and Remark \ref{rem3} gives the following {more general corollary to Theorem \ref{main}. }

\begin{corollary}\label{nice2}
Let $\ell\geq 5$ be prime, and let $m,k, k'$ be positive integers such $k'\equiv k \pmod{\ell^m}$ and $k'\equiv k \equiv -4 \pmod{\ell^{m-1}}$.  If for each $1\leq r\leq m$,
\[
p_k(\ell^rn-\delta_{k,\ell,r})\equiv 0 \pmod{\ell^r}
\]
for all $0\leq n < \frac{k \ell^r +2\ell+2}{24}$,
then there is a congruence at $(\ell^r, k', -\delta_{k',\ell, m})$ for all $1\le r\le m$. 
\end{corollary}
\begin{remark}
As in Remark \ref{rem3}, we can replace the conditions on 
$p_k(\ell^rn-\delta_{k,\ell,r})$ in Corollary \ref{nice2} by the appropriate conditions on $p_k(\ell^rn-\delta_{k,\ell,m})$ and the conclusion that there are congruences $(\ell^r, k', -\delta_{k',\ell, m})$ by congruences at 
$(\ell^r, k', -\delta_{k',\ell, r})$.
\end{remark}

In Section \ref{proof}, we prove Theorem \ref{main}, and in Section \ref{examples} we provide some examples of congruences for $p_k(n)$ proved by Theorem \ref{main}.  

\section{Preliminaries}\label{prelim}

Recall the ring of formal Laurent series in $q$ with integer coefficients defined by
\[
\mathbb{Z}\pq:= \left\{ f(q) = \sum_{n\geq n_0} a(n)q^n \mid n_0\in \mathbb{Z}, a(n)\in \mathbb{Z} \mbox{ for all }n\geq n_0 \right\}.
\]
The set of formal power series in $q$ with integer coefficients, denoted $\mathbb{Z}\bq$, is the subring of $\mathbb{Z}\pq$ obtained by fixing $n_0=0$ in the above definition.


\subsection{Reduction of power series modulo prime powers}\label{power}

\begin{proposition}\label{handy}
Let $f(q)\in \mathbb{Z}\pq$  (resp. $\mathbb{Z}\bq$).  For any integer $k\geq 0$, and prime $\ell$, 
\[
f(q^{\ell^k})^\ell = g(q^{\ell^{k+1}})+\ell\cdot h(q^{\ell^k}),
\]
where $g,h\in\mathbb{Z}\pq$ (resp. $\mathbb{Z}\bq$).
\end{proposition}

\begin{proof}
Fix an integer $k\geq 0$, and let $f(q) = \sum_{n\geq n_0}a(n)q^n$, with $a(n)\in\mathbb{Z}$.  Then $f(q^{\ell^k}) = \sum_{n\geq n_0} a(n) q^{n\ell^k}$, so that
\[
f(q^{\ell^k})^\ell  \equiv \sum_{\substack{n\geq n_0\\ \ell \nmid a(n)}}a(n) q^{n\ell^{k+1}} \pmod \ell.
\]
Thus the coefficients of any terms $q^m$ of $f(q^{\ell^k})^\ell$ where $m$ is not a multiple of $\ell^{k+1}$ must be divisible by $\ell$.  So we can write 
\[
f(q^{\ell^k})^\ell = g(q^{\ell^{k+1}})+\ell\cdot h(q^{\ell^k}),
\]
for series $g,h\in\mathbb{Z}\pq$.  
\end{proof}

\noindent The following lemma is used many times leading up to Theorem \ref{main}.

\begin{lemma}\label{l-divisibility}
Let $\ell$ be prime, and $f\in\mathbb{Z}\pq$ (resp. $\mathbb{Z}\bq$).  Then for any positive integer $m$, there exist $f_0,f_1,\ldots, f_m\in \mathbb{Z}\pq$ (resp. $\mathbb{Z}\bq$), such that
\[
f(q)^{\ell^m} = f_0(q^{\ell^m}) + \ell\cdot f_1(q^{\ell^{m-1}}) + \ell^2 \cdot f_2(q^{\ell^{m-2}}) + \cdots + \ell^{m-1}\cdot f_{m-1}(q^\ell)  + \ell^m\cdot f_m(q),
\]
\end{lemma}

\begin{proof}  Fix a prime $\ell$, and let $f\in\mathbb{Z}\pq$ with coefficients given by 
\[
f(q)=\sum_{n\geq n_0}a(n) q^n.
\]
We induct on $m$.  From Proposition \ref{handy} with $k=0$, we immediately obtain the case $m=1$.  Suppose now for an arbitrary positive integer $m$ that   
\[
f(q)^{\ell^m} = f_0(q^{\ell^m}) + \ell\cdot f_1(q^{\ell^{m-1}}) + \cdots + \ell^{m-1}\cdot f_{m-1}(q^\ell)  + \ell^m\cdot f_m(q),
\]  
for some $f_0,f_1,\ldots, f_m\in \mathbb{Z}\pq$.
 Then,
\[
f(q)^{\ell^{m+1}} =(f(q)^{\ell^m})^\ell = \left( f_0(q^{\ell^m}) + \ell f_1(q^{\ell^{m-1}}) +  \cdots + \ell^{m-1} f_{m-1}(q^\ell)  + \ell^m f_m(q) \right)^\ell.
\]
By the multinomial theorem, this yields that
\[
f(q)^{\ell^{m+1}} = \sum_{\substack{k_0+\cdots + k_m = \ell \\  k_i\geq 0}} \binom{\ell}{k_0, \dots ,k_m} \left[ f_0(q^{\ell^m}) \right]^{k_0} \left[ \ell f_1(q^{\ell^{m-1}}) \right]^{k_1} \cdots \left[ \ell^m f_m(q) \right]^{k_m}.
\]
We break this sum into cases based on the highest index $i$ for which $k_i\neq 0$.  In particular, letting
\[
F_i(q) := \sum_{\substack{k_0+\cdots+k_i =\ell \\ k_i\neq 0}}\binom{\ell}{k_0, \ldots, k_i} \left[ f_0(q^{\ell^m}) \right]^{k_0} \left[ \ell f_1(q^{\ell^{m-1}}) \right]^{k_1} \cdots \left[ \ell^i f_i(q^{\ell^{m-i}}) \right]^{k_i},
\]
for $0\leq i \leq m$, we obtain that
\[
f(q)^{\ell^{m+1}} = F_0(q) + F_1(q) + \cdots + F_m(q).
\]
We observe that $F_0(q)$ has only one term, $k_0=\ell$.  Thus $F_0(q) = f_0(q^{\ell^{m}})^\ell$.  Applying Proposition \ref{handy}, we see that 
\[
F_0(q) = g_0(q^{\ell^{m+1}})+\ell\cdot g_1(q^{\ell^m}),
\]
for some $g_0, g_1\in\mathbb{Z}\pq$.
For $1\leq i \leq m$, we separate the term $k_i=\ell$ to obtain
\[
F_i(q) =\left[ \ell^i f_i(q^{\ell^{m-i}})\right]^\ell + \sum_{\substack{k_0+\cdots+ k_i=\ell \\ 1\leq k_i \leq \ell-1}}\binom{\ell}{k_0, \ldots, k_i} \left[ f_0(q^{\ell^m}) \right] ^{k_0} \cdots \left[ \ell^i f_i(q^{\ell^{m-i}}) \right] ^{k_i}.
\]
We can see from above that $F_i(q)$ is a series in $q^{\ell^{m-i}}$.  In addition, each term is a multiple of $\ell^{i+1}$.  This is clear in the first term, since we have a factor of $\ell^{i\ell}$, and $\ell\geq 2$.  For the remaining terms we have a factor of $\ell^i$ appearing in $\left[ \ell^i f_i(q^{\ell^{m-i}}) \right]^{k_i}$ since $k_i\geq 1$, but in addition $\ell$ divides the multinomial coefficient $\binom{\ell}{k_0, \ldots, k_i}$, since $k_i \leq \ell-1$.  Thus we have for each $1\leq i \leq m$, that
\[
F_i(q) = \ell^{i+1}g_{i+1}(q^{\ell^{m-i}}), 
\]
where $g_{i+1}\in\mathbb{Z}\pq.$  Together, we have shown that
\[
f(q)^{\ell^{m+1}} =  g_0(q^{\ell^{m+1}})+\ell\cdot g_1(q^{\ell^m}) +  \ell^{2}g_{2}(q^{\ell^{m-1}}) + \cdots +  \ell^{m+1}g_{m+1}(q).
\]
\end{proof}

The following is a useful generalization of a lemma of Kiming and Olsson \cite{K-O} that gives an infinite family of congruences for each congruence proved using Theorem \ref{main}.  In particular, it yields Corollary \ref{nice2} as a consequence {of Corollary \ref{nice}}.  

\begin{lemma}\label{KOgeneralization}
Let $\ell\geq 5$ be prime.  For an integer $a$, and positive integers $m,k$, there is a congruence at $(\ell^r, k, a)$ for all $1\leq r\leq m$ if and only if there is a congruence at $(\ell^r, k+\ell^m, a)$ for all $1\leq r\leq m$.
\end{lemma}

\begin{proof}

We will show {more generally} that whenever there is a congruence at $(\ell^r,k,a)$ for all $1\leq r\leq m$, there is also a congruence at $(\ell^r,k+s\ell^m,a)$ for all $1\leq r\leq m$, for any {$s\in\mathbb{Z}$} for which $k+s\ell^m \geq 1$.  The lemma then follows by letting $s=1$ in the ``only if" direction, and $s=-1$ in the ``if" direction.  We prove this statement by induction on $m$.   

The case when $m=1$ {follows directly from} Lemma 1 of Kiming and Olsson in \cite{K-O}.  Suppose the claim holds {then} for a positive integer $m$ and assume there is a congruence at $(\ell^r,k,a)$ for all $1\leq r\leq m+1$.  {Fix an integer $s$} such that $k+s\ell^{m+1}\geq 1$.  Since $s\ell^{m+1}=(s\ell)\ell^m$, by the induction hypothesis there is a congruence at $(\ell^r, k+s\ell^{m+1}, a)$ for all $1\leq r\leq m$.  It remains to show there is a congruence at $(\ell^{m+1}, k+s\ell^{m+1}, a)$. 

Since {by assumption} we know there is a congruence at $(\ell^{m+1}, k, a)$, our goal is to relate $p_k(\ell^{m+1}n+a)$ and $p_{k+s\ell^{m+1}}(\ell^{m+1}n+a)$ modulo $\ell^{m+1}$.  Considering the generating functions of $p_k(n)$ and $p_{k+s\ell^{m+1}}(n)$, we see that
\[
\sum_{n=0}^\infty p_{k+s\ell^{m+1}}(n) q^n = \left( \sum_{n=0}^\infty p_k(n) q^n\right) 
\left( \left( \prod_{n=1}^\infty \frac{1}{1-q^n}\right)^{s} \right)^{\ell^{m+1}}.
\]
Thus, by Lemma \ref{l-divisibility},
\begin{multline*}
\sum_{n=0}^\infty p_{k+s\ell^{m+1}}(n) q^n = \\ \left( \sum_{n=0}^\infty p_k(n) q^n\right)\left( f_0(q^{\ell^{m+1}}) + \ell\cdot f_1(q^{\ell^{m}})  + \cdots + \ell^{m}\cdot f_{m}(q^\ell)  + \ell^{m+1}\cdot f_{m+1}(q) \right),
\end{multline*}
where $f_0,\ldots, f_{m+1}\in \mathbb{Z}\bq$. Note that we can apply Lemma \ref{l-divisibility} even if $s< 0$.

Define coefficients $c_i(n)\in\mathbb{Z}$ of $f_i(q^{\ell^{m+1-i}})$ by 
\[
f_i(q^{\ell^{m+1-i}}) = \sum_{n\geq 0}c_i(n)q^{n\ell^{m+1-i}}.
\]
Then, 
\begin{multline}\label{coeff}
\sum_{n=0}^\infty p_{k+s\ell^{m+1}}(n) q^n = \\ \left( \sum_{n=0}^\infty p_k(n) q^n\right) \left( \sum_{n\geq 0} c_0(n)q^{n\ell^{m+1}} + \ell \sum_{n\geq 0} c_1(n)q^{n\ell^{m}} + \cdots + \ell^{m+1} \sum_{n\geq 0} c_{m+1}(n)q^{n}  \right).
\end{multline}

For  {fixed integers $n\geq 0$, $a\geq 0$, and $0\leq i \leq m+1$}, the coefficient of $q^{n\ell^{m+1}+a}$ in the term 
\[
\ell^i\left(\sum_{j=0}^\infty p_k(j) q^j \right) \left( \sum_{j\geq 0} c_i(j) q^{j\ell^{m+1-i}} \right)
\]
is 
\[
\ell^i\sum_{j=0}^{n\ell^i+b_i} p_k\left( (\ell^i n - j)\ell^{m+1-i} + a \right)c_i(j),
\]
where  {$b_i=\lfloor \frac{a}{\ell^{m+1-i}}\rfloor$ and so} the $b_i$ satisfy $a = \ell^{m+1-i}b_i+a'$ for some $0\leq a'< \ell^{m+1-i}$.
 {By} equating the coefficients of $q^{n\ell^{m+1}+a}$ from the left and right sides of (\ref{coeff}), we have that for all $n\geq 0$,
\begin{equation}\label{sun}
p_{k+s\ell^{m+1}}(n\ell^{m+1}+a) = \sum_{i=0}^{m+1} \left( \ell^i\sum_{j=0}^{n\ell^i+b_i} p_k\left( (\ell^i n - j)\ell^{m+1-i} + a \right)c_i(j) \right).
\end{equation}
Since we are assuming that there is a congruence at $(\ell^r,k,a)$ for all $1\leq r\leq m+1$, for each $0\leq i\leq m$ we have that for any $0\leq j\leq {n\ell^i}+b_i$, 
\[
\ell^ip_k\left( (\ell^i n - j)\ell^{m+1-i} + a \right) \equiv 0 \pmod{\ell^{m+1}}.
\]
Combining this with 
(\ref{sun}) yields that
\[
p_{k+s\ell^{m+1}}(n\ell^{m+1}+a) \equiv 0 \pmod{\ell^{m+1}}
\]
for all $n\geq 0$, as desired.

\end{proof}

\subsection{Connection to modular forms on $\rm{SL}_2(\mathbb{Z})$}

Recall that a holomorphic modular form of integer weight $k$ on the modular group ${\rm SL}_2(\mathbb{Z})$ is a holomorphic function from the upper half plane $f:\mathcal{H}\rightarrow \mathbb{C}$, such that for all $z\in\mathcal{H}$, and $\left( \begin{smallmatrix} a & b \\ c & d \end{smallmatrix} \right) \in {\rm SL}_2(\mathbb{Z})$, 
\[
f\left(\frac{az + b}{cz + d}\right) = (cz+d)^kf(z).
\]
In addition, $f$ is holomorphic at $\infty$.  I.e., $f$ has a Fourier expansion of the form 
\[
f(z) = \displaystyle\sum_{n\geq 0}a(n)q^n, 
\]
where $q=e^{2\pi iz}$.  If $a(0)=0$, we say $f$ is a cusp form.  We denote by $M_k$ (resp. $S_k$) the finite dimensional complex vector space of holomorphic modular (cusp) forms of weight $k$.  For the reader unfamiliar with modular forms, we recommend \cite{Serre}, \cite{DiamondandShurman}, and \cite{Ono} for a nice introduction.

We now set $q=e^{2\pi iz}$ throughout.  The delta function $\Delta(z)$, defined by 
\[
\Delta(z) := q\prod_{n=1}^\infty (1-q^n)^{24},
\]
is a weight $12$ cusp form.  In addition, for even $k\geq 4$ the classical Eisenstein series $E_k$ defined by 
\[
E_k(z) := 1 + \frac{-2k}{B_k}\displaystyle\sum_{n\geq 1}\sigma_{k-1}(n)q^n,
\]
where $B_k$ is the $k^{th}$ Bernoulli number, and $\sigma_{k-1}(n) = \sum_{d|n}d^{k-1}$, lies in $M_k$.

For odd integers $k$, $M_k$ has dimension $0$.  When $k\geq 2$ is even, $M_k$ has dimension given by 
\begin{equation}\label{dim}
\dim(M_k) = \begin{cases}
    \lfloor \frac{k}{12} \rfloor & \text{ if } k\equiv 2 \pmod{12}, \\
   \lfloor \frac{k}{12} \rfloor +1  &  \text{ if } k\not\equiv 2 \pmod{12}.
\end{cases}
\end{equation}
A basis for $M_k$ can be given in terms of the Eisenstein series $E_4$ and $E_6$.  The following lemma is based on properties of Bernoulli numbers and can be found in \cite{Ono} (Lemma 1.22).

\begin{lemma}\label{E-equiv}
Let $k\geq 2$ be even, and $p$ prime.  If $k$ is divisible by $p-1$, then 
\[
E_k(z) \equiv 1 \pmod{p^{{\rm ord}_p (2k)+1}},
\]
where ${\rm ord}_p(2k)$ is the largest integer $n$ for which $p^n\mid 2k$.  
\end{lemma} 

The proof of Theorem \ref{main} relies on an elegant result of Choie, Kohnen, and Ono found in \cite{CKO}.  To state this result we need some additional terminology.  For a function $f(z)$ with Fourier expansion $f(z)=\sum_{n\geq n_0} a_n q^n$, define 
\[
const(f) = a_0.
\]
In addition, for a positive even integer $k$, define
\begin{equation}\label{E}
\tilde{E}_k = \begin{cases}
  1    & \text{ if } k\equiv 0 \pmod{12} , \\
   E_{14}    & \text{ if } k\equiv  2\pmod{12} ,\\
   E_4     &\text{ if } k\equiv  4\pmod{12} ,  \\
     E_6    & \text{ if } k\equiv  6\pmod{12} , \\
       E_4^2   &  \text{ if } k\equiv  8\pmod{12} ,\\
      E_4E_6     & \text{ if } k\equiv  10\pmod{12}  . \\
\end{cases}
\end{equation}
   
\begin{theorem}\label{CKO}(Choie, Kohnen, Ono)
Let $f\in M_{12n+14}$, $g\in M_k$, and $D(k)=\dim(M_k)$.  Then,
\[
const\left( \frac{f\cdot g}{\Delta^{n+D(k)}\cdot \tilde{E_k}} \right) = 0.
\]

\end{theorem}

\begin{definition}
Let $k,m$ be positive integers.  For all integers $n\geq 0$ we define $\tau_{k,m}(n)$ to be the $n$th Fourier coefficient of $\Delta(z)^{\delta_{k,\ell,m}}$.  I.e., 
\begin{equation}\label{tau}
\sum_{n=0}^\infty \tau_{k,m}(n) q^n = q^{\delta_{k,\ell,m}}\prod_{n=1}^\infty (1-q^n)^{k(\ell^{2m}-1)} = \Delta(z)^{\delta_{k,\ell,m}}.
\end{equation}
Note that $\tau_{k,m}(n)\in\mathbb{Z}$.
\end{definition}

The reason we are able to use Theorem \ref{CKO} to prove Theorem \ref{main} is due to an explicit relationship between $p_k(n)$ and $\Delta(z)$.  Notice that if we fix positive integers $k,m$ and a prime $\ell\geq 5$, we have from (\ref{tau}) that  
\[
\Delta(z)^{\delta_{k,\ell,m}} = \sum_{n=0}^\infty \tau_{k,m}(n) q^n = q^{\delta_{k,\ell,m}}\left(  \sum_{n=0}^\infty p_k(n)q^n \right) \left(\prod_{n=1}^\infty (1-q^n)\right) ^{k\ell^{2m}}.
\]
 For any $1\leq r \leq 2m$,
\begin{equation}\label{taueqn}
\sum_{n=0}^\infty \tau_{k,m}(n) q^n = \left(  \sum_{n=0}^\infty p_k(n-\delta_{k,\ell,m})q^n \right) \left( \left(\prod_{n=1}^\infty (1-q^n)\right) ^{k\ell^{2m-r}}\right) ^{\ell^r}.
\end{equation}
Applying Lemma \ref{l-divisibility} to $\prod_{n=1}^\infty (1-q^n)^{k\ell^{2m-r}}$, gives that 
\[
\sum_{n=0}^\infty \tau_{k,m}(n) q^n = \left(  \sum_{n=0}^\infty p_k(n-\delta_{k,\ell,m})q^n \right) \left( f_0(q^{\ell^{r}}) + \ell f_1(q^{\ell^{r-1}})  + \cdots + \ell^{r} f_{r}(q)  \right),
\]
for some $f_0,\ldots, f_{r}\in \mathbb{Z}\bq$.  For $0\leq i\leq r$, define coefficients $c_{r,i}(n)\in\mathbb{Z}$ by 
\[
f_i(q^{\ell^{r-i}}) = \sum_{n\geq 0}c_{r,i}(n)q^{n\ell^{r-i}}.
\]
Thus we have
\begin{multline*}
\sum_{n=0}^\infty \tau_{k,m}(n) q^n =  \left(  \sum_{n=0}^\infty p_k(n-\delta_{k,\ell,m})q^n \right) \cdot \\ \left( \sum_{n\geq 0} c_{r,0}(n)q^{n\ell^{r}} + \ell \sum_{n\geq 0} c_{r,1}(n)q^{n\ell^{r-1}} + \cdots + \ell^{r} \sum_{n\geq 0} c_{r,r}(n)q^{n}  \right).
\end{multline*}
The coefficient of $q^{n\ell^r}$ in the term 
\[
\left(  \sum_{n=0}^\infty p_k(n-\delta_{k,\ell,m})q^n \right)  \left( \ell^i\sum_{n\geq 0} c_{r,i}(n)q^{n\ell^{r-i}} \right)
\]
is 
\[
\ell^i\sum_{j=0}^{n\ell^i}p_k\left((n\ell^i-j)\ell^{r-i} - \delta_{k,\ell,m}\right) c_{r,i}(j).
\]
Thus we have for any $1\leq r\leq2 m$,
\begin{align}\label{moon}
\tau_{k,m}(n\ell^r) &= \sum_{i=0}^r \ell^i\sum_{j=0}^{n\ell^i}p_k\left((n\ell^i-j)\ell^{r-i} - \delta_{k,\ell,m}\right) c_{r,i}(j) \nonumber \\ 
& \equiv \sum_{i=0}^{r-1} \ell^i\sum_{j=0}^{n\ell^i}p_k\left((n\ell^i-j)\ell^{r-i} - \delta_{k,\ell,m}\right) c_{r,i}(j) \pmod{\ell^r}.
\end{align}

If we rewrite (\ref{taueqn}) as
\[
\sum_{n=0}^\infty p_k(n-\delta_{k,\ell,m})q^n  = 
\left( \sum_{n=0}^\infty \tau_{k,m}(n) q^n \right)  \left( \left(\prod_{n=1}^\infty \frac{1}{(1-q^n)} \right) ^{k\ell^{2m-r}}\right) ^{\ell^r}
\]
and let 
\begin{multline*}
 \left( \left(\prod_{n=1}^\infty \frac{1}{(1-q^n)} \right) ^{k\ell^{2m-r}}\right) ^{\ell^r} = \\
 \sum_{n\geq 0} b_{r,0}(n)q^{n\ell^{r}} + \ell \sum_{n\geq 0} b_{r,1}(n)q^{n\ell^{r-1}} + \cdots + \ell^{r} \sum_{n\geq 0} b_{r,r}(n)q^{n} ,
\end{multline*}
again using Lemma  \ref{l-divisibility}, then similarly we get that 
\begin{align}\label{moon2}
p_k(n\ell^r-\delta_{k,\ell, m}) &=\sum_{i=0}^r \ell^i\sum_{j=0}^{n\ell^i} \tau_{k,m}((n\ell^i-j)\ell^{r-i}) b_{r,i}(j) \nonumber \\
 & \equiv  \sum_{i=0}^{r-1} \ell^i\sum_{j=0}^{n\ell^i} \tau_{k,m}((n\ell^i-j)\ell^{r-i}) b_{r,i}(j)\pmod{\ell^r}.
\end{align}

We will use these facts in our proof of the following lemma, which generalize a proposition of Lachterman, Schayer, and Younger in \cite{LSY}. 

\begin{lemma}\label{m-1}
Let $\ell\geq 5$ be prime, $k,m$ and $B$ be positive integers, and $1\leq r \leq m$.  There is a congruence at $(\ell^s, k, -\delta_{k,\ell,m})$  for each $1\leq s \leq r$  if and only if
\[
\tau_{k,m}(\ell^sn)\equiv 0 \pmod{\ell^s}
\]
for all $n\geq 0$, for each $1\leq s\leq r$.  

Moreover if there is a congruence at $(\ell^s, k, -\delta_{k,\ell,m})$  for each $1\leq s \leq r$, and $p_k(\ell^{r+1}n-\delta_{k,\ell,m}) \equiv 0 \pmod{\ell^{r+1}}$ for all $n < B$, then $\tau_{k,m}(\ell^{r+1}n)\equiv  0 \pmod{\ell^{r+1}}$ for all $n<B$.

\end{lemma}

\begin{proof}

Suppose there is a congruence at $(\ell^s, k, -\delta_{k,\ell,m})$ for each $1\leq s\leq r$, and fix $1\leq s\leq r$, $n\geq 0$.  Then for each $0\leq i \leq s-1$, and $0\leq j\leq n\ell^i$ we have that
\[
p_k\left((n\ell^i-j)\ell^{s-i} - \delta_{k,\ell,m}\right) \equiv 0 \pmod{\ell^{s-i}}.
\]
Since $1\leq s\leq m$, by (\ref{moon}) it follows that 
\[
\tau_{k,m}(n\ell^s) \equiv 0 \pmod{\ell^s}
\]
for all $n\geq 0$.

The final sentence of the theorem follows from (\ref{moon}) as well, by considering $r+1$, and restricting values of $n$ to $n<B$.


The reverse implication follows from a similar argument utilizing (\ref{moon2}).
\end{proof}

\section{The proof of Theorem \ref{main}}\label{proof}

For each $n\in\mathbb{Z}$ we define the integers 
\[
w_{k,\ell,m}(n) = 12(n\ell^m-\delta_{k,\ell,m}) +2,
\]
when $k,m$ are positive integers, and $\ell\geq 5$ is prime.

\begin{proposition}\label{one}
Fix a prime $\ell\geq 5$, and positive integers $k,m$ such that $k\equiv -4 \pmod{\ell^{m-1}}$.  Then for any integer $n$, we have 
\[
w_{k,\ell,m}(n)  \equiv 0 \pmod{\ell^{m-1}}.
\]
\end{proposition}

\begin{proof}
Since $k\equiv -4 \pmod{\ell^{m-1}}$, we have $k(\ell^{2m}-1)\equiv 4 \pmod{\ell^{m-1}}$, and since $2$ is relatively prime to $\ell$,
\[
12\delta_{k,\ell,m} = \frac{k(\ell^{2m}-1)}{2}\equiv 2 \pmod{\ell^{m-1}}.
\]
Thus for any integer $n$,
\[
w_{k,\ell,m}(n) \equiv -12\delta_{k,\ell,m} + 2 \equiv 0 \pmod{\ell^{m-1}}.
\]
\end{proof}

\noindent In light of Proposition \ref{one}, for each $n\in \mathbb{Z}$, there is an integer $C_n$ such that $w_{k,\ell,m}(n) = C_n\ell^{m-1}$.  Moreover, $C_n$ must be even since $w_{k,\ell,m}(n)$ is even, and $\ell$ is odd.  Thus we define 
\[
K_n\in\{0,4,6,\ldots, \ell-3, \ell+1\}
\]
when $\ell\geq7$, and $K_n\in\{0,4,6\}$ when $\ell=5$, so that 
\[
K_n\equiv C_n \pmod{\ell-1}.
\]

\begin{proposition}\label{two}
Fix a prime $\ell\geq 5$, and positive integers $k,m$ such that $k\equiv -4 \pmod{\ell^{m-1}}$.  Then if $n \ge \frac{k \ell^m +2\ell+2}{24}$,
 we have that 
\[
s_n = \frac{w_{k,\ell,m}(n)-K_n\ell^m}{\ell^{m-1}(\ell-1)}
\]
is a positive integer.
\end{proposition}

\begin{proof}
From Proposition \ref{one} we see that $w_{k,\ell,m}(n)-K_n\ell^m \equiv 0 \pmod{\ell^{m-1}}$.  To conclude that $s_n\in\mathbb{Z}$, it suffices to show that $w_{k,\ell,m}(n)-K_n\ell^m \equiv 0 \pmod{\ell-1}$, since $\ell^{m-1}$ and $\ell-1$ are relatively prime.  This fact follows from the definition of $K_n$.  Also, $s_n$ is positive because 
$$
12 n\ell^m \ge \frac{\ell^m(k \ell^m +2\ell+2)}{2}
> \frac{k (\ell^{2m}-1)} {2} + \ell^m(\ell +1) -2
= 12 \delta_{k,\ell, m}  + \ell^m(\ell +1) -2
$$
where the first inequality follows from the assumption that 
 $n \ge \frac{k \ell^m +2\ell+2}{24}$. Thus,
\[
w_{k,\ell,m}(n) > \ell^m(\ell+1)\geq K_n\ell^m,
\]
and so $s_n>0$.
\end{proof}

We are now able to prove Theorem \ref{main}.  Suppose that there is a congruence at $(\ell^r,k,-\delta_{k,\ell,m})$ for all $1\leq r \leq m-1$.  Clearly if there is a congruence at $(\ell^m,k,-\delta_{k,\ell,m})$, then we have that 
\[
p_k(\ell^mn-\delta_{k,\ell,m})\equiv 0 \pmod{\ell^m}
\] 
for all $0 \le  n <  \frac{k \ell^m +2\ell+2}{24} $.  We now show that if 
\[
p_k(\ell^mn-\delta_{k,\ell,m})\equiv 0 \pmod{\ell^m}
\]
for all $0\le n < \frac{k \ell^m +2\ell+2}{24}$, then there is a congruence at $(\ell^m,k,-\delta_{k,\ell,m})$.

Define \[
g(z) = \Delta(z)^{\delta_{k,\ell,m}},
\]
and for any integer $n\geq \frac{k \ell^m +2\ell+2}{24}$, define
\[
f_n(z)= E_{K_n}(z)^{\ell^m}\cdot E_{\ell^{m-1}(\ell-1)}(z)^{s_n},
\]
where we make the convention that $E_{K_n}(z)=1$ if $K_n=0$.  Recall by Proposition \ref{two}, that $s_n$ is a positive integer. 

The weight of $f_n(z)$ is 
\[
K_n\ell^m+s_n\ell^{m-1}(\ell-1) = w_{k,\ell,m}(n).
\]
Thus $f_n(z)\in M_{12(n\ell^m-\delta_{k,\ell,m}-1) +14}$, and $g(z)\in M_{12\delta_{k,\ell,m}}$.  We have by (\ref{dim}) that $D(12\delta_{k,\ell,m})= \delta_{k,\ell,m} + 1$, and by (\ref{E}) that $\tilde{E}_{12\delta_{k,\ell,m}}= 1$.  Thus Theorem \ref{CKO} gives that
\[
const\left( \frac{f_n\cdot g}{\Delta^{n\ell^m}} \right) = 0.
\]

\noindent By Lemma \ref{E-equiv}, $E_{\ell^{m-1}(\ell-1)}(z)\equiv 1 \pmod{\ell^{m}}$.  Thus $f_n(z) \equiv E_{K_n}(z)^{\ell^m} \pmod{\ell^m}$, and so 
\begin{equation}\label{zero}
const\left( \frac{E_{K_n}(z)^{\ell^m}\cdot \Delta(z)^{\delta_{k,\ell,m}}}{\Delta(z)^{n\ell^m}} \right) \equiv 0 \pmod{\ell^m}.
\end{equation}

Consider the Fourier series in $\mathbb{Z}\pq$ of 
\[
\left( \frac{E_{K_n}(z)}{\Delta(z)^n} \right)^{\ell^m}.
\]
Since the Fourier series of $E_{K_n}(z)$ starts with  $1 + \cdots$, and $\Delta(z)^n = q^{n} + \cdots$, it follows that
\begin{equation}\label{expansion}
\left(  \frac{E_{K_n}(z)}{\Delta(z)^n} \right)^{\ell^m} = q^{-n\ell^m} + \cdots .
\end{equation}

By Lemma \ref{l-divisibility}, there are series $f_0, \ldots, f_m\in \mathbb{Z}\pq$ such that
\[
\left( \frac{E_{K_n}(z)}{\Delta(z)^n} \right)^{\ell^m} = f_{0,n}(q^{\ell^m}) + \ell\cdot f_{1,n}(q^{\ell^{m-1}}) + \cdots + \ell^m\cdot f_{m,n}(q).
\] 
In light of (\ref{expansion}), $f_{0,n}, \ldots, f_{m,n}$ can be chosen so that 
\[
f_{i,n}(q^{\ell^{m-i}}) = \sum_{N\geq -n\ell^i}c_{i,n}(N)q^{N\ell^{m-i}},
\]
where $c_{0,n}(-n) =1$ and $c_{i,n}(-n\ell^i)= 0$ if $i \ne 0$.
Thus,
\begin{multline*}
\left( \frac{E_{K_n}(z)}{\Delta(z)^n} \right)^{\ell^m} = \\
\sum_{N\geq -n}c_{0,n}(N)q^{N\ell^m} + \ell \sum_{N>-n\ell} c_{1,n}(N)q^{N\ell^{m-1}} + \cdots + \ell^m \sum_{N>-n\ell^m}c_{m,n}(N)q^{N}.
\end{multline*}
Then the constant term of $\left( \frac{E_{K_n}(z)}{\Delta(z)^n} \right)^{\ell^m}\cdot  \Delta(z)^{\delta_{k,\ell,m}} $ satisfies 
\[
 \sum_{i=0}^m \ell^i\cdot const \left[ \left(\sum_{N\geq-n\ell^i}c_{i,n}(N)q^{N\ell^{m-i}} \right)\left( \sum_{N\geq 0} \tau_{k,m}(N)q^N \right) \right] \equiv 0 \pmod{\ell^m}
\]
by (\ref{zero}).  For each $0\leq i \leq m$, 
\begin{multline}\label{tau3}
\ell^i\cdot const\left[  \left(\sum_{N\geq-n\ell^i}c_{i,n}(N)q^{N\ell^{m-i}} \right)\left( \sum_{N\geq0} \tau_{k,m}(N)q^N \right) \right] \\
= \ell^i\cdot  \sum_{j=0}^{n\ell^i}c_{i,n}(-j)\tau_{k,m}(j\ell^{m-i}).
\end{multline}
If $m>1$, using Lemma \ref{m-1} gives that for all $1\leq r \leq m-1$,
\begin{equation}\label{tau0}
\tau_{k,m}(\ell^rn)\equiv 0 \pmod{\ell^r},
\end{equation} 
for all $n\geq 0$.  But when $1\leq i\leq m$, (\ref{tau0}) gives that $\tau_{k,m}(j\ell^{m-i})\equiv 0 \pmod{\ell^{m-i}}$ for any $j\geq 0$.  It follows that (for any $m\geq 1$)
\begin{equation}\label{tau2}
\ell^i \cdot  \tau_{k,m}(j\ell^{m-i})\equiv 0 \pmod{\ell^{m}}.
\end{equation}

Taking $i=0$ we see that for any $n\geq \frac{k \ell^m +2\ell+2}{24}$,
\begin{equation}\label{almost}
\sum_{j=0}^{n}c_{0,n}(-j)\tau_{k,m}(j\ell^{m}) \equiv 0 \pmod{\ell^m}.
\end{equation}
However, if
\[
p_k(\ell^mn-\delta_{k,\ell,m})\equiv 0 \pmod{\ell^m}
\]
for integers $0\leq n < \frac{k \ell^m +2\ell+2}{24}$, then by the final remark of Lemma \ref{m-1} 
\[
\tau_{k,m}(\ell^mn)\equiv 0 \pmod{\ell^m}
\]
for $0\leq n < \frac{k \ell^m +2\ell+2}{24}$ as well.  Combining this with (\ref{almost}) we have that for all $n\geq 0$,
\begin{equation}\label{tau4}
\sum_{j=0}^{n}c_{0,n}(-j)\tau_{k,m}(j\ell^{m}) \equiv 0 \pmod{\ell^m}.
\end{equation}
 
Using the fact that $c_{0,n}(-n) =1$, we obtain by induction on $n$ that 
\[
\tau_{k,m}(n\ell^m)\equiv 0 \pmod{\ell^m}
\]
 for all $n\geq 0$. Indeed for $n = 0$ we have 
\[
c_{0,0}(-0)\tau_{k,m}(0)  \equiv 0 \pmod{\ell^m}
\]
and since $c_{0,0}(-0) =1$, we get that 
\[
\tau_{k,m}(0)  \equiv 0 \pmod{\ell^m}.
\]
Suppose that the $(n-1)$th case has been proven. Then (\ref{tau4} ) can be written as 
\begin{align}\label{moon3}
c_{0,n}(-n)\tau_{k,m}(j\ell^{m})
+ \sum_{j=0}^{n-1}c_{0,n}(-j)\tau_{k,m}(j\ell^{m}) 
 \equiv 0 \pmod{\ell^m}.
\end{align}
By induction
$$ 
\sum_{j=0}^{n-1}c_{0,n}(-j)\tau_{k,m}(j\ell^{m}) \equiv 0 \pmod{\ell^m}.
$$
Also, $c_{0,n}(-n) = 1$. Therefore by Equation (\ref{moon3}), we have 
$$
\tau_{k,m}(n\ell^{m}) \equiv 0 \pmod{\ell^m}
$$
for all $n\ge 0$ as desired. Combining this with (\ref{tau0}), we have for  all $1 \le r \le m$
$$
\tau_{k,m}(n\ell^{r}) \equiv 0 \pmod{\ell^r}
$$
for all $n\ge 0$.

Finally, by Lemma \ref{m-1} with $r=m$, we have a congruence at $(\ell^m,k,-\delta_{k,\ell,m})$, as desired.
\qed

\section{Examples of congruences for $p_k(n)$}\label{examples}

We now address the question of what types of congruences Theorem \ref{main} implies.  Using a theorem of Kiming and Olsson, we shall see that Theorem \ref{main} can be used to prove most congruences where $m=1$.  Note that the condition $k\equiv -4\pmod{{\ell}^{m-1}}$ is trivial for $m=1$.  Thus we have no restriction on the number of components, $k$.  In particular, the classical Ramanujan congruences for $k=1$, and $\ell=5,7,11$ all can be proven by Theorem \ref{main}, so it is in fact a generalization of the theorem of Lachterman, Schayer, and Younger in \cite{LSY}.

Kiming and Olsson define an \emph{exceptional} congruence to be one of the form
\[
p_k({\ell}n+a)\equiv 0 \pmod{{\ell}}
\]
where $1\leq k \leq \ell-1$ and $k\not\in\{{\ell}-1,\,{\ell}-3\}$.  They prove the following theorem in \cite{K-O}.
\begin{theorem}(Theorem 1 in \cite{K-O})
Let ${\ell} \ge 5$ be a prime number. If 
\[
p_k({\ell}n+a)\equiv 0\pmod{\ell}
\]
is an exceptional congruence, then $k$ is odd and $24a\equiv k \pmod {\ell}.$
\end{theorem}

We can see that Theorem \ref{main} applies to every exceptional congruence by noticing that in this case
\[
\delta_{k,{\ell},1} = \frac{k{\ell}^2-k}{24} \equiv -a \pmod{\ell}.
\] 

We also find that Theorem \ref{main} can be used to prove a non-exceptional conguence.  First we recall the following result of Gandhi \cite{Gandhi}.

\begin{theorem}(Gandhi)
If ${\ell}\geq 5$ is prime, and $a=\frac{{\ell}^2-1}{8}$, then 
\[
p_{{\ell}-3}({\ell}n+a)\equiv 0  \pmod {\ell}.
\]
\end{theorem}

If $a=\frac{{\ell}^2-1}{8}$, then since $\ell ^2 - 1$ is divisible by $24$ we have $a=3b$ for some positive integer $b$.  Thus,
\[
\delta_{\ell-3, \ell, 1} = \frac{(\ell-3)(\ell^2-1)}{24} = \frac{\ell - 3}{3}\cdot a = (\ell -3)b \equiv -a \pmod{\ell},
\]
and so Theorem \ref{main} applies to these non-exceptional cases of Gandhi as well.

\subsection{All Applicable Congruences for Small Primes}

We conclude with a list of all infinite families of congruences that can be proven with Theorem \ref{main} for ${\ell}\le 13$ and $m\le 2$.  

\begin{theorem}\label{congruences} For all integers $r, n\geq 0$,
\begin{equation}
p_{2+5r}(5n+3)\equiv 0 \pmod{5}
\end{equation}
\begin{equation}
p_{1+5r}(5n+4)\equiv 0 \pmod{5}
\end{equation}

\begin{equation}
p_{1+7r}(7n+5)\equiv 0 \pmod{7}
\end{equation}
\begin{equation}
p_{4+7r}(7n+6)\equiv 0 \pmod{7}
\end{equation}

\begin{equation}
p_{8+11r}(11n+4)\equiv 0 \pmod{11}
\end{equation}
\begin{equation}
p_{1+11r}(11n+6)\equiv 0 \pmod{11}
\end{equation}
\begin{equation}
p_{3+11r}(11n+7)\equiv 0 \pmod{11}
\end{equation}
\begin{equation}
p_{5+11r}(11n+8)\equiv 0 \pmod{11}
\end{equation}
\begin{equation}
p_{7+11r}(11n+9)\equiv 0 \pmod{11}
\end{equation}

\begin{equation}
p_{10+13r}(13n+8)\equiv 0 \pmod{13}
\end{equation}
and

\begin{equation}
p_{11+5^2r}(5^2n+14)\equiv 0 \pmod{5^2}
\end{equation}
\begin{equation}
p_{6+5^2r}(5^2n+19)\equiv 0 \pmod{5^2}
\end{equation}
\begin{equation}
p_{1+5^2r}(5^2n+24)\equiv 0 \pmod{5^2}
\end{equation}

\begin{equation}
p_{95+11^2r}(11^2n+9)\equiv 0 \pmod{11^2}
\end{equation}
\begin{equation}
p_{84+11^2r}(11^2n+64)\equiv 0 \pmod{11^2}
\end{equation}
\begin{equation}
p_{7+11^2r}(11^2n+86)\equiv 0 \pmod{11^2}
\end{equation}
\begin{equation}
p_{29+11^2r}(11^2n+97)\equiv 0 \pmod{11^2}
\end{equation}
\begin{equation}
p_{51+11^2r}(11^2n+108)\equiv 0 \pmod{11^2}
\end{equation}
\begin{equation}
p_{73+11^2r}(11^2n+119)\equiv 0 \pmod{11^2}.
\end{equation}

\end{theorem}

\end{document}